\documentclass[reqno,A4paper]{amsart}
\usepackage{amsmath, amsthm, amscd, amsfonts, amssymb, graphicx, color}
\usepackage{cite}
\usepackage[colorlinks]{hyperref}
 \usepackage{mathtools}
\makeatletter
\@namedef{subjclassname@2020}{%
  \textup{2020} Mathematics Subject Classification}
\makeatother
\newtheorem{theorem}{Theorem}[section]
\newtheorem{cor}[theorem]{Corollary}
\newtheorem{lem}[theorem]{Lemma}

\theoremstyle{definition}
\newtheorem{defn}[theorem]{Definition}

\newtheorem{rem}[theorem]{Remark}
\numberwithin{equation}{section}

\newcommand{\A}{\mathcal{A}}

\newcommand{\V}{\mathcal{V}}

\begin{document}
\title[]{SOME COMBINATORIAL CONCEPTS NEAR AN IDEMPOTENT}
\author[A. Pashapournia \and M. A. Tootkaboni \and D. Ebrahim Bagha]%
{A. Pashapournia* \and M. A. Tootkaboni**  \and D. Ebrahimi Bagha*}

\newcommand{\acr}{\newline\indent}

\address{\llap{*\,}Department of Mathematics\acr
Faculty of Sciences\acr
 Islamic Azad University\acr
 Central Tehran Branch\acr
 Tehran, Iran}
\email{ali154067@gmail.com}

\address{\llap{**\,}Department of Pure Mathematics\acr
                    Faculty of Mathematical sciences\acr
                    University of Guilan\acr
                    Rasht, Iran}
\email{tootkaboni@guilan.ac.ir}

\address{\llap{*\,}Department of Mathematics\acr
Faculty of Sciences\acr
 Islamic Azad University\acr
 Central Tehran Branch\acr
 Tehran, Iran}
\email{e\underline{ }bagha@yahoo.com}

\keywords{$wap$-compactification, Idempotent, $J$-set, $C$-set, Central sets Theorem.\\
 $^\dag$Corresponding Author}
\begin{abstract}
A small part of real line which is very close to zero has rich combinatorial properties.
The aim of this paper is to express and then prove some locally combinatorial concepts near a virtual idempotent by considering  the $wap$-compactification of a semitopological semigroup $S$. The $wap-$compactification of a semitopological semigroup $S$, is denoted by $S^w$, the collection of all ultrafilters near an idempotent $\eta\in S^w$ forms a compact subsemigroup of $\beta S_d$, where $S_d$ denotes $S$ as discrete space.
\end{abstract}
\maketitle
\section{\textbf{Introduction}}

  Furstenberg defined the concept of central subset of the natural numbers $\mathbb{N}$ in \cite{1}. Also
  he stated that for every finite partition of natural numbers, one of the cells contains a central set. The following theorem is the original Central Sets Theorem:
 \begin{theorem} (\cite{1} Proposition 8.21) Let $l\in\mathbb{N}$ and for each $i \in\{1, 2,\ldots, l\}$, let $\langle y_{i,n}\rangle _{n=1}^{\infty}$ be
 a sequence in $\mathbb{Z}$. Let $C$ be a central subset of $\mathbb{N}$. Then there exist sequences $\langle  a_n\rangle _{n=1}^{\infty}$ in
 $\mathbb{N}$ and $\langle H_n\rangle _{n=1}^{\infty}$ in $P_f(\mathbb{N})$(=the collection of all nonempty finite subsets of $\mathbb{N}$), such that\\

 \indent {(1)} for all $n$, $\max H_n< \min H_{n+1}$, and\\
\indent {(2)} for all $F\in P_f(\mathbb{N})$ and all $i\in\{1, 2,\ldots, l\}$, $\sum_{n\in F}(a_n+\sum_{t\in H_n} y_{i,t}) \in C.$
 \end{theorem}
 In \cite{2}, Furstenberg and Katznelson developed a technique so it was used later  to provide another proof of the above Theorem (See \cite{3}). The following theorem is a new version of Central Sets Theorem assuming that $S$ is a commutative semigroup.
 \begin{theorem}(\cite{hindbook} Theorem 14.8.4)\label{theorem21}
For a central subset $A$ of a commutative semigroup $(S,+)$, there are functions $\alpha : P_{f}(^\mathbb{N}S)\rightarrow S$ and
$H : P_{f}(^\mathbb{N}S)\rightarrow P_{f}(S)$ such that\\

 \indent {(1)} if $\emptyset\neq F\subsetneq G\in P_{f}(^\mathbb{N}$, then $\max H(F)<\min H(G)$, and\\
\indent {(2)} if $m\in \mathbb{N}$, $\{X_i\}_{i=1}^m\subseteq P_{f}(^\mathbb{N}S)$, $X_{1}\subsetneq X_{2}\subsetneq \cdots\subsetneq X_{m}$,
and  $f_{i}\in X_{i}$ for each $i\in\{1,2,\ldots,m\}$, then $\sum_{i=1}^{m}(\alpha(X_{i})+\sum_{t\in H(X_{i})}f_{i}(t)) \in A$.
\end{theorem}
Each set satisfying the conclusion of the Central Sets Theorem is called $C$-set. Let $A$ be a subset of
commutative semigroup $(S,+)$, whenever
$F \in P_{f}(^\mathbb{N}S)$, there exist $a\in S$ and $H\in P_{f}(\mathbb{N})$ such that
\[
a+\sum_{t\in H}f(t)\in A
\]
for each $f\in F$, we say that $A$ is a $J-$set.

Let $(S,+)$ be an arbitrary infinite semigroup. The spectrum of $l^\infty(S)$ is called the Stone-$\check{C}$ech of $S$ and is denoted by $\beta S$.
It is the biggest semigroup compactification of $S$. In other words,\\
(a) right translation for every $x\in\beta S$ and left translation for each
$s\in S$ are continuous,\\
(b) the topological center of $\beta S$ is a dense subset of $\beta S$, and \\
(c) if $(X,\phi)$ is a semigroup compactification of $S$, then there exists an onto homomorphism $\psi:\beta S\rightarrow X$ such that
$\psi\circ \varepsilon=\phi$, where $\varepsilon:S\rightarrow \beta S$ is evaluation map.
For more details of semigroup of compactification see \cite{Analyson}.

 The subset
 $L$ is said to be a left ideal of $S$ if $s+l\in L$ for every $s\in S$ and every $l\in L$. And $L$ is a minimal left ideal of $S$ if
 for each left ideal $J$ of $S$, whenever $J\subseteq L$ one gets $J=L$. The minimal
 right ideal is defined analogously. An element $x\in S$ is an idempotent if
 $x+x=x$ and $E(S)$ denotes the collection of all idempotents in $S$. If $S$ is a compact Hausdorff right topological semigroup, then
   the minimal ideal, $K(S)$, exists and $K(S)$ is a disjoint union of minimal right ideals as well as a disjoint union of minimal left ideals. Every idempotent in $K(S)$ is called minimal idempotent.  Every element of a minimal idempotent is called central set. Also a subset $A$ of $S$ is called central$^*$ set if it is a member of
every minimal idempotent. See \cite{hindbook}.

 An ultrafilter $p$ on $S=(0,+\infty)$ is called near zero if $(0,\varepsilon)\in p$ for all $\varepsilon>0$. Let
 \[
 0^+=\{p\in \beta S_d:\forall\epsilon>0\,\,\, (0,\epsilon)\in p\}.
 \]
 Then $(0^+,+)$ is a closed subsemigroup of $\beta S_d$.
In \cite{Hin-Lead}, N. Hindman and I. Leader stated some combinatorial results in real numbers near zero. They introduced also new combinatorial applications of the sets which are central near zero. For more detail see \cite{De-Hin, De-Hin-Str}.

Let $S\subseteq (0,\infty)$ be an additive dense subsemigroup. The set of all
function $f:\mathbb{N}\rightarrow S$ such that $lim_{n\rightarrow\infty}f(n)=0$ is
denoted by $\mathcal{T}_0$. The set $A\subseteq S$ is a central set near zero if
there exists an idempotent $p$ contained in smallest ideal of $0^+(S)$ with $A\in p$.\\
In \cite{ba}, authors gave definitions of  the $J$-set and the $C$-set near zero and stated the Central sets Theorem near zero.
As a natural consequence, has been proved
\[
J_0(S)=\{p\in 0^+(S) : \forall A \in p, \ A\  is \  a \  J-set\  near\  zero \}.
\]
is a two side ideal of $0^+(S)$, and every element of idempotent of $J_0(S)$ is $C$-set.
\begin{rem}
Since $(\mathbb{N},+)$ is a commutative semigroup, so the set $\mathbb{N}^w$ as
$wap$-compactification of the natural numbers is a commutative compact semitopological semigroup. By Lemma \ref{lem2.17}, there exists
a surjective continuous homomorphism $\pi:\beta \mathbb{N}\rightarrow\mathbb{N}^w$ such that $\pi(p)\subseteq p$ for every
$p\in\beta \mathbb{N}$ and $\pi(K(\beta \mathbb{N}))=K(\mathbb{N}^w)$. By Theorem 21.19 and Corollary 2.40 in \cite{hindbook},
$K(\mathbb{N}^w)$ is a compact topological group so $|E(K(\mathbb{N}^w))|=1$. Moreover, if $\eta$ is the minimal idempotent of $\mathbb{N}^w$, therefore $\eta\subseteq p$ for every $p\in E(K(\beta \mathbb{N}))$, and so $E(K(\beta \mathbb{N}))\subseteq\eta^*$.
\end{rem}
The concept of $J$-set near $\eta$ is partition regular, see Definition 3.1, and Lemma 4.9. So
\[
J_\eta(\mathbb{N})=\{p\in\beta \mathbb{N}:\forall A\in p,\mbox{ $A$ is $J$-set near $\eta$.}\}.
\]
is a nonempty set and by Theorem 4.10 is a two sided ideal. As a consequence, we will have $E(K(\beta \mathbb{N}))\subseteq J_\eta(\mathbb{N})\subseteq \eta^*$,
and so every central$^*$ set is a $J$-set near $\eta$, and so is every central set.
Now let
$$M=\{p\in\beta \mathbb{N}:\forall A\in p,\mbox{ $A$ is an additive central set}\}.$$
Then, according to Theorem 16.24 in \cite{hindbook}, we get $M=cl(E(K(\beta\mathbb{N},+))\subseteq\eta^*$ is a left ideal of $(\beta\mathbb{N},\cdot)$. And so, by Theorem 16.26.1, every central$^*$ set in $(\mathbb{N},\cdot)$
 is a central set in $(\mathbb{N},+)$, i.e. every multiplicative central$^*$ set is an additive central set, and so is a $J$-set near $\eta$.

Assume that
 \[
 \Delta=\left\{A\subseteq\beta\mathbb{N}:\forall A\in p, \limsup_{n\rightarrow\infty}\frac{\text{card}(A\cap\{1,\cdots,n\})}{n}>0\right\}.
 \]
 Then every combinatorial rich ultrafilters belong to $\eta^*$, where combinatorial rich ultrafilters is a multiplicative idempotent $p\in M\cap\Delta\cap K(\beta\mathbb{N},\cdot)$. By Theorem 17.1 in \cite{hindbook}, there is a combinatorial rich ultrafilter. For some properties of combinatorial rich ultrafilters see Theorem 17.3 in \cite{hindbook}. So it seems that $\eta^*$ has rich combinatorial properties.
 In this paper, we just concentrate on extension of concepts $J$-set and $C$-set near an idempotent in $wap$-compactification of a semitopological semigroup.

As a result of  the above points, we have two versions of the Central sets Theorem on $(0,+\infty)$,
Global Central sets Theorem and Local Central sets Theorem. The zero for the semigroup of $(0,+\infty)$ is a virtual
idempotent, but it can easily be used for other semigroup. It seems that, the idempotents in the
$wap$-compactification of a semitopological semigroup can act as a virtual idempotent.
Therefore, the question arises as to whether like the virtual idempotent of $0$, the Local Central sets Theorem and related
concepts can be defined for the virtual idempotent of a subsemigroup.

In what follows, we give and prove some locally combinatorial concepts near a virtual idempotent. For
this purpose, we consider the $wap$-compactification of a semitopological semigroup $S$.
When $S$ is a discrete semigroup, the $Lmc$-compactification and $\beta S$ are topologically isomorphism, and also if $S$
is a commutative semigroup then $wap$-compactification of $S$ is commutative, so we fucus on
$wap$-compactification of a semitopological semigroup $S$.

According to \cite{Akbarii}, the $Lmc$-compactification is expressed as a space of \break$e$-ultrafilters.
However, in similar ways, we can describe the $wap$-compact-\break ification of a semitopological semigroup as
a space of $e$-ultrafilters. For this purpose, in brief, for a semitoplogical semigroup of $S$ in Section 2, the $wap$-compactification of
$S$ is described as a space of $e$-ultrafilters, which in continues play an essential role. In Section 3, we state some combinatorial concepts
near an idempotent in $wap-$compactification of a commutative semigroup $S$ and the next section, we concentrate on noncommutative semigroup.


\section{\textbf{Preliminary}}

Let $(S,+)$ be a semitopological semigroup (i.e., right and left translations are continuous).
 The collection of all bounded complex valued continuous functions on $S$ with uniform norm is a $C^*$-algebra
 and is denoted by $\mathcal{C}\mathcal{B}(S)$. A weakly almost periodic function, $f$, is a member of
  $\mathcal{C}\mathcal{B}(S)$ such that $\{R_{s}f: s\in S\}$ is weakly
 relatively compact subset of $\mathcal{CB}(S)$. $wap(S)$ denotes the weakly almost periodic functions on $S$. $(\varepsilon, S^w)$ as the $wap$-compactification of a semitopological semigroup is the universal semitopological semigroup compactification of $S$, see \cite{Analyson}.

 In \cite{Akbarii}, $Lmc$-compactification has been characterized as a space
 of $e$-ultarfilters. We know that the $Lmc$-compactification of a discrete semigroup is same the Stone-$\check{C}$ech compactification. We restate some definitions and concepts from \cite{Akbarii}, because we need to describe $wap$-compactification as a space of $e$-ultrafilters.

 For every complex valued function $f$ on $S$, $Z(f)=\{s\in S:f(s)=0\}$ is called a zero set. Let $Z(wap(S))=\{Z(f):f\in wap(S)\}$.

We say that $\mathcal{A}\subseteq Z(wap(S))$ is a $z$-filter on $wap(S)$ ($z-$filter) if,

 \indent {$(i)$} $\emptyset\notin \mathcal{A}$,

 \indent {$(ii)$}  for every $A,B\in \mathcal{A}$, implies that $A\bigcap B\in\mathcal{A}$,

  \indent {$(iii)$} if $A\in \mathcal{A}$, $B\in Z(wap(S))$ and $A\subseteq
B$ then $B\in \mathcal{A}$.

For a complex valued function $f$ and $\epsilon >0$, define
$E_{\epsilon}(f)=\{s\in S:|f(s)|\leq\epsilon\}$.
For $I$ as a collection of complex valued functions on $S$, define $E(I)=\{E_{\epsilon}(f):f\in
I,\epsilon>0\}$. Also, for any
collection $\A\subseteq Z(wap(S)$, let
\[
E^{-}(\A)=\{f\in wap (S):E_{\epsilon}(f)\in\A \mbox{ for each }\epsilon >0\}.
\]
If $\A$ is a $z-$filter, we say that $\A$ is an $e-$filter if
$E(E^{-}(\A))=\A$. For every $z-$ultrafilter $\A$, $E(E^{-}(\A))$ is an $e-$ultrafilter.

For a Hausdorff semitopological semigroup $S$. By Theorem 3.8 in \cite{Akbarii}, it is proved that
$\{p: p \mbox{ is an }e-\mbox{ultrafilter.}\}$ and $S^w$ are topologically isomorphic. Therefore,
\[
\{A^{\dag}=\{p\in \mathcal{E}(S):A\in p\}:A\in Z(wap(S))\}
\]
forms a basis for a topology on $S^w$. Each $a\in S$, define $e(a)=\{\ E_{\epsilon}(f):f(a)=0, \epsilon>0\}$ is an $e$-ultrafilter.  Also, for
$A\in Z(wap(S))$ and $x\in S$, we have $A\in x+p$ if and only
if $\lambda_x^{-1}(A)\in p$, see Lemma 3.9 in \cite{Akbarii}.

Now, we ready define ultrafilters near an idempotent. For $\eta\in S^w$, we define
\[
\eta^*=\{u\in\beta S_d:\eta\in\bigcap_{A\in u}cl_{S^w} A\}.
\]
We know that if $p,q\in\beta S$ and $A\subseteq S$, then we have $A\in p+q$ if and only if
$\{x\in S:-x+A\in q\}\in p$, where $-x+A=\{y\in S:x+y\in A\}$. See \cite{hindbook}.

\begin{lem}\label{lem2.17}
Let $(S,+)$ be a semitopological semigroup.\\
$(a)$ For $\eta\in S^w$, $\eta^*=\{p\in\beta S_d:\eta\subseteq p\}$.\\
$(b)$ Let $A\subseteq S$. Then $x\in cl_{S^w}A$ if and only if $cl_{\beta S_d}A\cap x^*\neq \emptyset.$\\
$(c)$ For each $u,v\in S^w$, $u^* + v^*\subseteq (u+v)^*$.\\
$(d)$ Let $\eta\in E( S^*)$, then $\eta^*$ is a compact subsemigroup of $\beta S_d$.
\end{lem}
\begin{proof}
(a) It is obvious.

(b) Let $x\in cl_{S^w}A$, so for each $U\in x$ we will have $U\cap A\neq\emptyset$. Therefore $x\cup \{A\}$
has the finite intersection property, this implies that there exists an ultrafilter $p$ such that $x\cup \{A\}\subseteq p$.
So $cl_{\beta S_d}A\cap x^*\neq \emptyset$.

Conversely, let $p\in cl_{\beta S_d}A\cap x^*$. therefore $A\in p$ and $x\subseteq p$.
This implies that for each $U\in x$ we will have $U\cap A\neq\emptyset$. So $x\in cl_{S^w}A$.

(c) Let $p\in x^*$ and $q\in y^*$. We prove $p+q\in (x+y)^*$. In order to we show that
$x+y\subseteq p+q$. So let $U\in x+y$, then there are $\epsilon>0$ and $f$ in $wap(S)$ such that
$U=E_{\epsilon}(f)$ and $E_{\delta}(y,f)=\{t\in
S:\lambda^{-1}_{t}(E_{\delta}(f))\in y\}\in x$ for each $\delta>0$. Now let $\delta=\epsilon$, so
\[
E_{\epsilon}(y,f)=\{t\in
S:\lambda^{-1}_{t}(E_{\epsilon}(f))\in y\subseteq q\}\in x\subseteq p.
\]
This implies that $U=E_{\epsilon}(f)\in p+q$.

(d) It is obvious that $\eta^*$ is a subsemigroup of $\beta S_d$. Now let $p\in \beta S_d\setminus \eta^*$.
Pick $U\in \eta\setminus p$. Then $cl_{\beta S_d}(S\setminus U)$ is a neighborhood of $p$ which
\[
cl_{\beta S_d}(S\setminus U)\cap \eta^*=\emptyset.
\]
Therefore $\eta^*$ is closed subset of $\beta S_d$.
\end{proof}
\begin{lem}
Let $(S,+)$ be a semitopological semigroup.\\
a) For each $p\in \beta S_d$, there exists a unique $e_p\in S^{w}$ such that $e_p\subseteq p$.\\
b) $\pi:\beta S_d\rightarrow S^{w}$ by $\pi(p)=e_p$ is well defined.\\
c) $\pi:\beta S_d\rightarrow  S^{w}$ is continuous homomorphism. In particular, for $p,q\in\beta S_d$, $\pi(p)+\pi(q)\subseteq p+q$.
\end{lem}
\begin{proof}
It is obvious.
\end{proof}

\begin{defn}
Let $(S,+)$ be a semitopological semigroup and  $\eta\in E(S^*)$. A subset $ A $ of  $ S $ is piecewise syndetic near
$\eta$ if and only if $ K(\eta^*)\cap cl_{ \beta S}(A)  \neq \emptyset $.
\end{defn}

\section{\textbf{ $J$-sets and $C$-sets near an idempotent}}

Let $(S,+)$ be a commutative semitopological semigroup. We define concepts of $J$-set and
$C$-set near a virtual idempotent $\eta$ in $S^w$.

For $B\subseteq \mathbb{N}$, the upper density of $B$ is defined by
 \[
 \overline{d}(B)=\limsup_{n\rightarrow \infty}\dfrac{|B\cap \{1,2,\cdots ,n\}|}{n}.
 \]

Let $\eta\in S^w$ and $f$ be a sequence in $S$.
We say that $\overline{d}-\lim_{n\in \mathbb{N}}f(n)=\eta$ if
for every  $U\in \eta$, $\overline{d}(\{n: f(n)\notin U\})=0$. In this paper,
\[
\mathcal{T}_\eta=\{f:\mathbb{N}\rightarrow S:\overline{d}-\lim_{n\in \mathbb{N}}f(n)=\eta\}.
\]

 Let $f,g\in \mathcal{T}_{x}$, and set $A=\{n\in \mathbb{N}:f(n)\notin U\}$ and $B=\{n\in \mathbb{N}:g(n)\notin U\}$.
 Then $\overline{d}(A)=\overline{d}(B)=0$, therefore $\overline{d}(A\cup B)=0$. By Exercise 3.11(c) in \cite{density},
 $\overline{d}(A\cup B)=1-\underline{d}((A\cup B)^c)$ implies that $\overline{d}(A^c\cap B^c)=1$. So
 $\overline{d}(f^{-1}(U)\cap g^{-1}(U))=1$ for each $U\in x$.

\begin{defn}{\label{j1}}
Let $A\subseteq S$ and $x\in cl_{S^w}A$. $A\subseteq S$ is called $J-$set near $x$ if for every
 $F\in P_f(\mathcal{T}_x)$ and for each $U\in x$ there are $a\in U$ and
$H\in P_f(\mathbb{N})$ such that
\[
a+\sum_{t\in H}f(t)\in A
\]
for each $f\in F$.
\end{defn}
\begin{lem}
If $\eta\in E(S^*)$ and $W\in\eta$, then $W$ is a $J$-set near $\eta$.
\end{lem}
\begin{proof}
Pick $k\in \mathbb{N}$ and let $F=\{f_1,\cdots,f_l\}\in P_f(\mathcal{T}_\eta)$.
Since
\[
\overline{d}(\bigcap_{f\in F}f^{-1}(V))=1
\]
 for each $V\in\eta$, so
by definition of $\mathcal{T}_\eta$, $A=\bigcap_{g\in F}g(\bigcap_{f\in F}f^{-1}(V))$
is infinite. Now pick $p_1,\cdots, p_n\in cl_{\beta S_d}A\cap \eta^*$. By Theorem 1.6.13 in \cite{Engel}, so there exist nets
$$\{\gamma_{\alpha_1}\}, \cdots,\{\gamma_{\alpha_n}\}$$
in $\mathbb{N}$ and $\{a_U\}_{U\in \eta}$ such that $a_U\in U$ $a_U\rightarrow \eta$ and
$f_i(\gamma_{\alpha_j})\rightarrow \eta$ for each $i\in\{1,\cdots, l\}$ and $j\in\{1,\cdots, l\}$ in $S^w$. Then
\[
\big(a_U+\sum_{i=1}^k f_1(\gamma_{\alpha_i}),\cdots,a_U+\sum_{i=1}^k f_l(\gamma_{\alpha_i})\big)\rightarrow \overline{\eta}=(\eta,\cdots,\eta)
\]
in $\times_{i=1}^lS^w$.
Now for $W\in\eta$, $\times_{i=1}^lW\in\times_{i=1}^l\eta$. So there exist
$H\in P_f(\mathbb{N})$ and $a_U\in U\subseteq W$ such that
\[
\big(a_U+\sum_{t\in H} f_1(t),\cdots,a_U+\sum_{t\in H} f_l(t)\big)\in \times_{i=1}^lW.
\]
This implies that $W$ is a $J$-set.
\end{proof}
\begin{lem}\label{lemma1}
Let $A$ be a subset of commutative semigroup $S$, and $\eta\in E(S^*)$. Then for every $m\in\mathbb{N}$, every
$F\in P_f(\mathcal{T}_\eta)$, and for each $U\in \eta$  there are $a\in U$ and
$H\in P_f(\mathbb{N})$ such that $\min H>m$ and  $a+\sum_{t\in H}f(t)\in A$ for each $f\in F$.
\end{lem}
\begin{proof}
See Lemma 14.8.2 in \cite {hindbook}.
\end{proof}
\begin{theorem}\label{theorem1}
For commutative semitopological semigroup $S$, let $A\subseteq S$.
Then every piecewise syndetic set near $\eta$ is a $J$-set near $\eta$.
\end{theorem}
\begin{proof}
Let $A\subseteq S$ be a piecewise syndetic set near $\eta$, and let $F=\{f_1,\cdots,f_l\}\in P_f(\mathcal{T}_\eta)$.
Let $Y=\times_{t=1}^{l}\eta^*\subseteq \times_{t=1}^{l}\beta S_d$.
So $Y$ is a compact right topological semigroup and for every $\vec{s}\in\times_{t=1}^{l}S$,
 $\lambda_{\vec{s}}$ is continuous, see  Theorem 2.22 in \cite{hindbook}. Pick $i\in\mathbb{N}$ and $U\in \eta$. Define

\begin{align*}
 I_{i,U}&=\\
 \big\{\big(a+\sum_{t\in H}f_1(t),\cdots,a+\sum_{t\in H}f_l(t)\big):a\in  U,
   H\in & P_f(\mathbb{N}), \mbox{ and }\min H>i\big\}\bigcap\times_{t=1}^{l}U
\end{align*}
and let $E_{i,U}=I_{i,U}\cup\{(a,\cdots,a):a\in  U\}$.

Let $E=\bigcap_{i\in\mathbb{N},U\in e}\overline{E_{i,U}}$ and let $I=\bigcap_{i\in\mathbb{N},U\in e}\overline{I_{i,U}}$.
Since $I_{i,U}\subseteq E_{i,U}\subseteq\times_{t=1}^{l}U$ for each $U\in \eta$,
$E\subseteq Y$ and $I\subseteq E$.

Now pick $p,q\in E$. We prove
$p+q\in E$, (i.e. $E$ is subsemigroup). Also, $p+q\in I$ if $p\in I$ or $q\in I$, (i.e. $I$ is two-side ideal). For $U\in \eta$,
so $p+q$ is interior point of $W=cl_{\beta S_d}  U$. Pick $i\in\mathbb{N}$, since right translation $\rho_q$ is continuous, there exists
open set $V$ such that $p\in V$ and
$V+q\subseteq W$. Now choose $\vec{x}\in E_{i,U}\cap V$ with $\vec{x}\in I_{i,U}$ if $p\in I$. Then for some $a\in  U$ and $H\in P_f(\mathbb{N})$ with $\min H>i$, $\vec{x}=(a+\sum_{t\in H}f_1(t),\cdots,a+\sum_{t\in H}f_l(t))$, whenever $x\in I_{i,U}$. Now pick $j=\max H$. Otherwise, let $j=i$. So choose an open set $Q$ such that $q\in Q$ and $\vec{x}+Q\subseteq W$, because left translatio $\lambda_{\vec{x}}$ is continuous. Pick $\vec{y}\in E_{j,U}\cap W$ with $\vec{y}\in I_{j,U}$ if $q\in I$.
Then $\vec{x}+\vec{y}\in E_{i,U}\cap W$ and if either $p\in I$ or $q\in I$, then $\vec{x}+\vec{y}\in I_{i,U}\cap W$.

Since $K(Y)=\times_{t=1}^{l}K(\eta^*)$, see Theorem 2.23 in \cite {hindbook}. Choose $p\in K(\eta^*)\cap \overline{A}$.
So $\overline{p}=(p,\cdots,p)\in K(Y)$. We show that $\overline{p}\in E$. Therefore, let $Z$ be an open neighborhood of
$\overline{p}$, let $i\in\mathbb{N}$, and choose $C_1,\cdots,C_l\in p$ such that $\times_{t=1}^{l}\overline{C_t}\subseteq Z$. For
$a\in\bigcap_{t=1}^lC_t$. Then $\overline{a}=(a,\cdots,a)\in W\cap E_{i,U}$. Thus
$\overline{p}\in K(Y)\cap E$ and consequently $K(Y)\cap E\neq\emptyset$. So we have that
$K(E)=K(Y)\cap E$ and so $\overline{p}\in K(E)\subseteq I$. Therefore $I_{1,U}\cap \times_{t=1}^{l}\overline{A}\neq\emptyset$ for each
$U\in \eta$, so choose $\vec{z}\in I_{1,U}\cap\times_{t=1}^{l} \overline{A}$ and $a\in  U$ and $H\in P_f(\mathbb{N})$ such that
\[
\vec{z}=\big(a+\sum_{t\in H}f_1(t),\cdots,a+\sum_{t\in H}f_l(t)\big).
\]
\end{proof}

\begin{theorem}\label{th1}
Let $S$ be a commutative semitopological semigroup and $\eta$ be a idempotent in $S^*$. Let $A$ be a central subset of $S$ near $\eta$. Then for $U\in \eta$,
there exist functions $\alpha_U:P_f(\mathcal{T}_\eta)\rightarrow S$ and $H_U:P_f(\mathcal{T}_\eta)\rightarrow P_f(\mathbb{N})$ such that

$(1)$ $\alpha_{U}(F)\in U$ for each $F\in P_f(\mathcal{T}_\eta)$,

$(2)$ for $F,G\in P_f(\mathcal{T}_\eta)$ and $F\subsetneq G$, implies that $\max H_U(F)<\min H_U(G)$, and

$(3)$ whenever $m\in \mathbb{N}$, $\{X_i\}_{i=1}^m\subseteq  P_f(\mathcal{T}_\eta)$, $X_1\subsetneq X_2\subset\cdots\subsetneq X_m$, and
for every $i\in \{1,2,\cdots , m\}$, $f_i\in X_i$, we have
\[
\sum_{i=1}^m\big(\alpha_U(X_i)+\sum_{t\in H_U(X_i)}f_i(t)\big)\in A.
\]
\end{theorem}
\begin{proof}
Choose a minimal idempotent $p$ of $\eta^*$ such that $A\in p$. Let $A^*=\{x\in A:-x+A\in p\}$, so $A^*\in p$.
Therefore, $x\in A^*$, implies that $-x+A^*\in p$, see Lemma 4.14 in \cite {hindbook}.

We define $\alpha_U(F)\in S$ and $H_U(F)\in P_f(\mathbb{N})$ for $F\in P_f(\mathcal{T}_\eta)$ and $U\in \eta$. By induction on $|F|$ satisfying the following statements:

(1) $\alpha_U(G)\in U$ for each $G\in \eta$,

(2) if $F,G\in P_f(\mathcal{T}_\eta)$ and $F\subsetneq G$, then $\max H_U(F)<\min H_U(G)$, and

(3) whenever $m\in \mathbb{N}$, $\{X_i\}_{i=1}^m\subseteq  P_f(\mathcal{T}_\eta)$, $X_1\subsetneq X_2\subset\cdots \subsetneq X_m$, and
for each $i\in \{1,2,\cdots , m\}$, $f_i\in X_i$, one has
\[
\sum_{i=1}^m\big(\alpha_U(X_i)+\sum_{t\in H_U(X_i)}f_i(t)\big)\in A^*.
\]

Assume that $F=\{f\}$. As $A^*$ is piecewise syndetic near $\eta$, choose for $U\in \eta$, $a\in S\cap U$ and $L\in P_f(\mathbb{N})$ such that
$a+\sum_{t\in L}f(t)\in A^*$. Define $\alpha_U(\{f\})=a$ and
$H_U(\{f\})=L$.

Assume that $|F|>1$, for every proper subsets $G$ of $F$ and for each $U\in \eta$,  $\alpha_U(G)$ and $H_U(G)$ have been defined.
For $U\in\eta$, define $K_U=\bigcup\{H_U(G):\emptyset\neq G\subsetneq F\}$ and let $m=\max K_U$. Define
\begin{align*}
M_U=\big\{\sum_{i=1}^n\big(\alpha_U(X_i)+&\sum_{t\in H_U(X_i)}f_i(t)\big):n\in\mathbb{N},\\
&\emptyset\neq X_1\subsetneq\cdots \subsetneq X_n\subsetneq F,
\,and\,\,\,\,\{f_i\}_{i=1}^n\in\times_{i=1}^n X_i\big\}.
\end{align*}
Therefore $M_U$ is finite set and by (3), $M_U\subseteq A^*$. Assume that
$B=A^*\cap\bigcap_{x\in M_U}(-x+A^*)$, so $B\in p$ and so pick $a\in S\cap U$ and $L\in
P_f(\mathbb{N})$ such that $a+\sum_{t\in L}f(t)\in B$ for each $f\in F$.
Let $\alpha_U(F)=a$ and $H_U(F)=L$.

The hypothesis (1) is clear. Since $\min L>m$, the hypothesis (2) is satisfied. To verify (3), choose $U\in \eta$ and
$n\in \mathbb{N}$, let $\emptyset\subsetneq X_1\subset\cdots \subsetneq X_n=F$, and $\{f_i\}_{i=1}^n\in\times_{i=1}^{n} X_i$.
If $n=1$, then $\alpha_U(X_1)+\sum_{t\in H_U(X_1)}f_1(t)=a+\sum_{t\in L}f_1(t)\in B\subseteq A^*$. So let $n>1$ and
define $y=\sum_{i=1}^{n-1}\big(\alpha_U(X_i)+\sum_{t\in H_U(X_i)}f_i(t)\big)$. Therefore
$y\in M_U$ so $a+\sum_{t\in L}f_1(t)\in B\subseteq(-y+A^*)$ and thus
$\sum_{i=1}^n\big(\alpha_U(X_i)+\sum_{t\in H_U(X_i)}f_i(t)\big)=y+a+\sum_{t\in L}f_1(t)\in A^*$ as required.
\end{proof}
\begin{defn}\label{defn1}
For a commutative semitopological semigroup $S$ let $\eta\in E(S^*)$. $A\subseteq S$ is called $C$-set near $\eta$ if
for $U\in \eta$,
there are functions $\alpha_U:P_f(\mathcal{T}_\eta)\rightarrow S$ and $H_U:P_f(\mathcal{T}_\eta)\rightarrow P_f(\mathbb{N})$ such that

(1) $\alpha_U(F)\in U$, for each $F\in P_f(\mathcal{T}_\eta)$,

(2) if $F,G\in P_f(\mathcal{T}_\eta)$ and $F\subsetneq G$, then $\max H_U(F)<\min H_U(G)$ and

(3) whenever $m\in \mathbb{N}$, $\{X_i\}_{i=1}^m\subseteq  P_f(\mathcal{T}_\eta)$, $X_1\subsetneq X_2\subset\cdots \subsetneq X_m$, and
for each $i\in \{1,2,\cdots , m\}$, $f_i\in X_i$, one has
\[
\sum_{i=1}^m\big(\alpha_U(X_i)+\sum_{t\in H_U(X_i)}f_i(t)\big)\in A.
\]
\end{defn}

\section{\textbf{Noncommutative version }}

In this section, we assume that $(S,+)$ is noncommutative semigroup. We state the concepts of $J$-set and $C$-set near an idempotent. Define
\[
\Phi=\{f\in ^\mathbb{N}\mathbb{N}:\forall n\in\mathbb{N}\,\,f(n)< n\}.
\]

\begin{defn}
For an arbitrary semitopological semigroup $S$ and for $m\in \mathbb{N}$, we define
\begin{multline*}
  \V_m=\{\times_{i=1}^mH_i\in P_f(\mathbb{N})^m:\mbox{ if }m>1,\, 1\leq t\leq m-1,\\
  \mbox{ then }\max H_t<\min H_{t+1} \}, \\
 \end{multline*}
\[
\mathcal{J}_m=\{\times_{i=1}^mt(i)\in\mathbb{N}^m:t(1)<\cdots<t(m)\},
\]
and
\[
S^m_U=S^m\cap U^m
\]
for $U\in \eta$.
\end{defn}
\begin{defn}
For an arbitrary semitopological semigroup $S$ and for $\eta\in E(S^*)$.

 Given $m\in \mathbb{N}$, $U\in \eta$, $a\in S_U^{m+1}$, $t\in\mathcal{J}_m$, and $f\in \mathcal{T}_\eta$, define
\[
x(m,a,t,f)=\sum_{j=1}^m\big(a(j)+f(t(j))\big)+a(m+1).
\]
\end{defn}
\begin{defn}{\label{j2}}
For an arbitrary semitopological semigroup $S$, $\eta\in E(S^*)$, $A\subseteq S$, and $\eta\in cl_{S^w}A$.

(a) $A$ is called $J$-set near $\eta$ if for each $F\in P_f(\mathcal{T}_\eta)$ and for each $U\in \eta$ there
exist $m\in \mathbb{N}$, $a\in S_U^{m+1}$, and $t\in \mathcal{J}_m$ such that for each $f\in F$, $x(m,a,t,f)\in A$.

b) $J_\eta(S)=\{p\in \eta^*:\forall A\in p, A\mbox{ is a }J-\mbox{set near }\eta\}.$
\end{defn}

\begin{lem}\label{lemma5.10}
For a commutative semitopological semigroup $S$, $\eta\in E(S^*)$, $A\subseteq S$, and $\eta\in cl_{S_w}A$. Then the follwing statemets equivalent:

(a) $A$ is a $J$-set near $\eta$, by Definition \ref{j1}.

(b) $A$ is a $J$-set near $\eta$, by Definition \ref{j2}.
\end{lem}
\begin{proof}
It is obvious that (b) implies (a).\\
If (a) is true. Pick $F\in P_f(\mathcal{T}_\eta)$, $U\in \eta$, $c\in \mathbb{N}$ and for
$f\in F$, define $g_f\in \mathcal{T}_\eta$ by $g_f(n)=f(n+c)$. Pick $b\in S$ and
$H\in P_f(\mathbb{N})$ such that for each $f\in F$, $b+\sum_{t\in H}g_f(t)\in
A$. Let $m=|H|$, and let $t=\big(t(1),\cdots,t(m)\big)$ enumerate $H$ in increasing order. Define $a(1)=b$
and for $j\in\{2,\cdots,m+1\}$, define $a(j)=c$. This complete the proof.
\end{proof}
\begin{defn}{\label{c2}}
For an arbitrary semitopological semigroup $S$, $\eta\in E(S^*)$, $A\subseteq S$, and $\eta\in cl_{S^w}A$. We say
$A$ is $C$-set near $\eta$ if  for each $U\in \eta$, there are $m_U:P_f(\mathcal{T}_\eta)\rightarrow \mathbb{N}$,
$\alpha_U\in\times_{F\in P_f(\mathcal{T}_\eta)}S_U^{m_U(F)+1}$,
and $\tau_U\in\times_{F\in P_f(\mathcal{T}_\eta)}\mathcal{J}_{m_U(F)}$ such that

(1) if $F,G\in P_f(\mathcal{T}_\eta)$ and $F\subsetneq G$ then $\tau_U(F)(m_U(F))<\tau_U(G)(1)$ for each $U\in \eta$, and

(2) whenever $n\in \mathbb{N}$, $X_1,\cdots, X_n\in P_f(\mathcal{T}_\eta)$,
$X_1\subsetneq X_2\subset\cdots \subsetneq X_n$, and for each
$i\in\{1,\cdots, n\}$, $f_i\in X_i$, one has
\[
\sum_{i=1}^nx(m_U(X_i),\alpha_U(X_i),\tau_U(X_i),f_i)\in A.
\]
\end{defn}
\begin{lem}\label{lemma5.11}
Let $S$ be a commutative semitopological semigroup and $\eta\in E(S^*)$,
let $A\subseteq S$, and $\eta\in cl_{S^w}A$. Then Definitions \ref{defn1} and \ref{c2} are equivalent.
\end{lem}
\begin{proof}
It is obvious that Definition \ref{c2} implies Definition \ref{defn1}. \\
Now let Definition \ref{defn1} be true. Pick $\alpha_U:P_f(\mathcal{T}_\eta)\rightarrow S$ and
$H_U:P_f(\mathcal{T}_\eta)\rightarrow P_f(\mathbb{N})$ for each $U\in \eta$
as guaranteed by $(1)$ and $(2)$. Now pick $c\in \mathbb{N}$ and for $f\in \mathcal{T}_\eta$ define $g_f\in\mathcal{T}_\eta$ by
$g_f(s)=f(s+ c)$, for $s\in \mathbb{N}$. For $F\in P_f(\mathcal{T}_\eta)$, we define
inductively on $|F|$ a set $K(F)\in P_f(\mathcal{T}_\eta)$ such that

$(1)$ $\{g_f:f\in F\}\subseteq K(F)$ and

$(2)$ if $\emptyset\neq G\subsetneq F$, then $K(G)\subsetneq K(F)$.

If $F=\{f\}$, let $K(F)=\{g_f\}$. Now let $|F|>1$ and $K(G)$ has been defined for all proper nonempty
subsets of $F$. Pick
 \[
 h\in \mathcal{T}_\eta\setminus\bigcup\{K(G):\emptyset\neq G\subsetneq F\}
 \]
  and
let
\[K(F)=\{h\}\cup\{g_f:f\in F\}\cup\bigcup\{K(G):\emptyset\neq G\subsetneq F\}\]
Now for each $U\in \eta$, we define
  $m_U:P_f(\mathcal{T}_\eta)\rightarrow \mathbb{N}$, $\alpha_{U}^\prime\in\times_{F\in P_f(\mathcal{T}_\eta)}S^{m_U(F)+1}$,
and $\tau_U\in\times_{F\in P_f(\mathcal{T}_\eta)}\mathcal{J}_{m_U(F)}$.
Let $F\in P_f(\mathcal{T}_\eta)$ be given and let $m_U(F)=|H_U(K(F))|$. Define
$\alpha_{U}^\prime (F)\in S^{m_U(F)+1}$, for $j\in\{1,2,\cdots,
m_U(F)+1\}$, $\alpha_{U}^\prime (F)(j)=\alpha_U(K(F))$ if $j=1$ and
$\alpha_{U}^\prime (F)(j)=c$ if $j>1$. Let $\tau_U(F)=(\tau_U(F)(1),\cdots,
\tau(F)(m_U(F)))$ enumerate $H_U(K(F))$ in increasing order. We need to show that

(a) if $F, G\in P_{f}(\mathcal{T}_\eta)$ and $F\subsetneq G$, then $\tau_U(F)(m_U(F))<\tau_U(G)(1)$ for each $U\in \eta$, and

(b) whenever $n\in \mathbb{N}$, $X_1,\cdots, X_n\in P_f(\mathcal{T}_\eta)$,
$X_1\subsetneq X_2\subset\cdots \subsetneq X_n$, and for each
$i\in\{1,\cdots, n\}$, $f_i\in X_i$, one has
\[
\sum_{i=1}^nx(m_U(X_i),\alpha_{U}^\prime(X_i),\tau_U(X_i),f_i)\in A.
\]
 To verify (a), let $F, G\in P_{f}(\mathcal{T}_\eta)$ with
$F\subsetneq G$, then $K(F)\subsetneq K(F)$, and so
 \[
 \tau_U(F)(m_U(F))=\max H_U(K(F))<\min H_U(K(G))=\tau_U(G)(1).
 \]
To verify (b), let $n\in \mathbb{N}$, $X_1,\cdots, X_n\in P_f(\mathcal{T}_\eta)$,
$X_1\subsetneq X_2\subset\cdots \subsetneq X_n$, and for each
$i\in\{1,\cdots, n\}$, let $f_i\in X_i$. Then $K(X_1)\subset
K(X_2)\subset\cdots \subsetneq K(X_n)$, and for each $f_i\in X_i$,
$g_{f_{i}} \in K(X_{i})$ so
$\sum^{n}_{i=1}(\alpha_U(K(X_{i})))+\sum_{t\in H_U(K(X_{i}))}g_{f_{i}}(t))\in A$ and\\
\begin{align*}
\sum^{n}_{i=1} & (\alpha_U(K(X_{i}))+ \sum_{t \in H_U(K(X_{i}))}
g_{f_{i}}(t))\\
&=\sum^{n}_{i=1}(\alpha_U(K(X_{i}))+\sum^{m_U(X_{i})}_{j=1}(f_{i}(\tau_U(X_{i})(j))+c))\\
&=\sum^{n}_{i=1}(\sum^{m_U(X_{i})}_{j=1}(\alpha_{U}^\prime(X_{i})(j)+(f_{i}(\tau_U(X_{i})(j)))+\alpha_{U}^\prime(X_{i})(m_U(X_{i})+1))\\
& =\sum^{n}_{i=1}x(m_U(X_i),\alpha_{U}^\prime(X_i),\tau_U(X_i),f_i).
\end{align*}
\end{proof}
\begin{lem}
Let $S$ be a commutative semitopological semigroup, $\eta\in E(S^*)$, $A\subseteq S$, and $\eta\in cl_{S^w}A$. Let $A$ be a $J$-set near $\eta$ in $S$, then for each
$F\in P_{f}(\mathcal{T}_\eta)$, each $U\in \eta$ and each $n\in\mathbb{N}$, there exist $m\in
\mathbb{N}$, $a\in S_U^{m+1}$, and $y\in \mathcal{J}_m $ such that $y(1)>n$ and for each $f\in F$, $x(m,a,y,f)\in A$.
\end{lem}
\begin{proof}
Pick $F\in P_f(\mathcal{T}_\eta)$, $U\in e$ and $n\in \mathbb{N}$. For each $f\in F$ define $X_f\in\mathcal{T}_\eta$  for
$u\in \mathbb{N}$, by $X_f(u)=f(u+n)$. Pick $m\in \mathbb{N}$, $a\in S_U^{m+1}$ and $t\in\mathcal{J}_m$ such that
for each $f\in F$, $x(m,a,t,X_f)\in A$. Define $y\in\mathcal{J}_m$ by $y(i)=n+t(i)$ for $i\in\{1,2,\cdots,m\}$.
Then $y(1)>1$ and for each $f\in F$, $x(m,a,y,f)\in A$.
\end{proof}
\begin{lem}\label{lemma2}
For an arbitrary semitopological semigroup $S$, $\eta\in E(S^*)$, $A\subseteq S$, and $\eta\in cl_{S^w}A$. Pick $U\in \eta$, and let
$m,r\in \mathbb{N}$, let $a\in S_U^{m+1}$,
let $t\in \mathcal{J}_{m}$, and for each $y\in \mathbb{N}$, let $c_{y}\in
S_U^{r+1}$ and $z_{y}\in \mathcal{J}_{r}$ be a such that for each $y\in
\mathbb{N}$, $z_{y}(r)< z_{y+1}(1)$. Then there exist
$u\in \mathbb{N}$, $d\in S_U^{u+1}$, and $q\in \mathcal{J}_{u}$ such that for
each $f\in \mathcal{T}_\eta$,
\[
(\sum_{j=1}^{m}a(j)+ x(r,c_{t}(j),z_{t}(j),f))+a(m+1)=x(u,d,q,f).
\]
\end{lem}
\begin{proof}
The proof is similar to Lemma 14.14.5 in \cite{hindbook}.
\end{proof}

\begin{lem}\label{lemma6}
Let $S$ be an arbitrary semitopological semigroup. Then the property of $J$-set near an idempotent of $E(S^*)$ is partition regular.
\end{lem}
\begin{proof}
See Lemma 14.14.6 in \cite{hindbook}.
\end{proof}
\begin{theorem}
Let $S$ be an arbitrary semitopological semigroup. Let $\eta\in E(S^*)$, $A\subseteq S$,
and $\eta\in cl_{S^w}A$. Then $J_\eta(S)$ is a compact two sided ideal of
$\eta^*$.
\end{theorem}
\begin{proof}
See Theorem 14.14.4 in \cite{hindbook}.
\end{proof}
\begin{theorem}\label{theorem6}
Let $S$ be a semitopological semigroup, $\eta\in E(S^*)$, $A\subseteq S$, and $\eta\in cl_{S^w}A$, $A\subseteq S$, and $\eta\in cl_{S^w}A$. Then $\overline{A}\cap
J_\eta(S)\neq \emptyset$ if and only if $A$ is a $J$-set near $\eta$.
\end{theorem}
\begin{proof}
The necessity is trivial. By Lemma \ref{lemma6}, $J$-sets are partition regular.
So, if $A$ is a $J$-set near $\eta$, by Theorem 3.11 in \cite{hindbook} , there is some $p\in \beta S$
such that $A\in p$ and for every $B\in p$, $B$ is a $J$-set near $\eta$.
\end{proof}
\begin{cor}
Let $S$ be a semitopological semigroup, $\eta\in E(S^*)$, $A\subseteq S$,  $\eta\in cl_{S^w}A$, and $A$ be a piecewise syndetic near $\eta$
subset of $S$. Then $A$ is a $J$-set near $\eta$.
\end{cor}
\begin{proof}
Since $A$ is piecewise syndetic near $\eta$, $\overline{A}\cap K(\eta^*)\neq \emptyset$. Therefore $K(\eta^*)\subseteq J_\eta(S)$ implies that
$\overline{A} \cap J_\eta(S) \neq \emptyset$ so by Theorem \ref{theorem6}, $A$ is a $J$-set near $\eta$.
\end{proof}

\begin{theorem}\label{the2}
Let $S$ be a semitopological semigroup, $\eta\in E(S^*)$, $A\subseteq S$, and $\eta\in cl_{S^w}A$. If there is an idempotent in $\overline{A}\cap
J_\eta(S)$, then $A$ is a $C$-set near $\eta$.
\end{theorem}
\begin{proof}
See Theorem 14.14.9 in \cite{hindbook}.
\end{proof}
\begin{cor}
Let $S$ be a semitopological semigroup, $\eta\in E(S^*)$, $A\subseteq S$, $\eta\in cl_{S^w}A$, and $A\subseteq S$ be a central set near $\eta$ in $S$. Then $A$ is
a $C$-set near $\eta$.
\end{cor}
\begin{proof}
It is obvious.
\end{proof}

\begin{lem}\label{lemma3.21}
Let $R$ be a set, let $(D,\leq)$ be a directed set, and let $S$ be a semitopological subsemigroup of $(T,+)$.
Let $\{T_i\}_{i\in D}$ be a decreasing family of nonempty subsets of $S$ such that

$1)$ $\eta\in cl_{T^w}T_i$,

$2)$ $\bigcap_{i\in D}T_i=\emptyset$, and

$3)$ for each $i\in D$ and each $x\in T_i$ there is some $j\in D$ such that $x+T_j\subseteq T_i$.

 Let
$T=\bigcap_{i\in D}cl_{\beta S_d}T_i$. Then $T$ is a compact subsemigroup of $\eta^*(S)$. Let $\{E_i\}_{i\in D}$ and
$\{I_i\}_{i\in D}$ be decreasing families of nonempty subsets of $\times_{t\in J}S$ with the following
properties:

$(a)$ for each $i\in D$, $I_i\subseteq E_i\subseteq \times_{t\in J} T_i$,

$(b)$ for each $i\in D$ and each $\vec{x} \in I_i$ there exists $j\in D$ such that $\vec{x}+E_j\subseteq I_i$,
and

$(c)$ for each $i \in D$ and each $\vec{x}\in E_i\setminus I_i$ there exists $j\in D$ such that
$\vec{x}+ E_j\subseteq E_i$ and $\vec{x} + I_j\subseteq I_i$.

Let $Y=\times_{t\in J}\eta^*(S)$, let $E=\bigcap_{i\in D}cl_YE_i$, and let $I=\bigcap_{i\in D}cl_Y I_i$. Then $E$ is
a subsemigroup of $\times_{t\in J}T$ and $I$ is an ideal of $E$. If, in addition, either

$(d)$ for each $i \in D$, $T_i=S$ and $\{a\in S:\overline{a}\notin E_i\}$ is not piecewise syndetic near $\eta$, or

$(e)$ for each $i\in D$ and each $a\in T_i$ , $\overline{a}\in E_i$,

then given
any $p\in K(T)$, one has $\overline{p}\in E\cap K(\times_{t\in J}T)=K(E)\subseteq I$.
\end{lem}
\begin{proof}
See Lemma 14.9 in \cite{hindbook}.
\end{proof}
\begin{theorem}
Let $S$ be a semitopological semigroup, $\eta\in E(S^*)$, $A\subseteq S$, and $\eta\in cl_{S^w}A$. Then $A$ is a $C$-set near $\eta$
if and only if there is an idempotent in $\overline{A}\cap J_\eta(S)$.
\end{theorem}
\begin{proof}
The proof is similar to Theorem 14.15.1 in \cite{hindbook}.
\end{proof}
\bibliographystyle{alpha}

\end{document}